\documentclass{jams-l}
\pdfoutput=1

\usepackage[colorlinks,citecolor=red,pagebackref,hypertexnames=false]{hyperref}
\usepackage{amssymb}
\usepackage{amsthm}
\usepackage{color}
\usepackage{esint}
\usepackage{bbm}
\usepackage{graphicx}
\usepackage[T1]{fontenc} 
\usepackage{doi}

\def\bH{{\mathbb{H}}}

\def\ve{\epsilon} 

\DeclareMathOperator{\diam}{diam}
\def\dim{\mathop\mathrm{dim}} 					
\def\dist{\mathop\mathrm{dist}} 						

\newcommand{\ps}[1]{\left( #1 \right)}

\newcommand{\ck}[1]{\left\{#1 \right\}}
\newcommand{\av}[1]{\left| #1 \right|}

\newcommand{\floor}[1]{\left\lfloor #1 \right\rfloor}

\newcommand{\fr}[2]{\frac{#1}{#2}}

\def\XXint#1#2#3{{\setbox0=\hbox{$#1{#2#3}{\int}$ }
\vcenter{\hbox{$#2#3$ }}\kern-.58\wd0}}

\def\grad{\nabla}

\newtheorem{theorem}{Theorem}[section]
\newtheorem{lemma}[theorem]{Lemma}

\theoremstyle{definition}

\theoremstyle{remark}
\newtheorem{remark}[theorem]{Remark}

\numberwithin{equation}{section}

\newcommand{\R}{\mathbb{R}}

\newcommand{\Z}{\mathbb{Z}}

\newcommand{\chara}{\mathbbm{1}}

\newcommand\blfootnote[1]{%
  \begingroup
  \renewcommand\thefootnote{}\footnote{#1}%
  \addtocounter{footnote}{-1}%
  \endgroup
}

\newcommand{\dD}{\mathcal{D}}
\newcommand{\HD}{{\mathsf{HD}}}
\newcommand{\LD}{{\mathsf{LD}}}
\usepackage[left=3.7cm, right=3.7cm]{geometry}
\numberwithin{equation}{section}
\theoremstyle{plain}

\newtheorem{corollary}[theorem]{Corollary}

\newtheorem{claim}[theorem]{Claim}

\title[Necessary condition]{Necessary condition for the $L^2$ boundedness of the Riesz transform on Heisenberg groups}

\author[D. D\k{a}browski]{Damian D\k{a}browski}
\address{Damian D\k{a}browski\newline\indent Departament de Matem\`atiques, Universitat Aut\`onoma de Barcelona; Barcelona Graduate School of Mathematics (BGSMath)\newline\indent Edifici C Facultat de Ci\`encies, 08193 Bellaterra, Barcelona, Catalonia, Spain.}
\email{ddabrowski ``at'' mat.uab.cat}
\author[M. Villa]{Michele Villa}
\address{Michele Villa \newline\indent School of Mathematics, University of Edinburgh, JCMB, Kings Buildings,
Mayfield Road, Edinburgh,
EH9 3JZ, Scotland.}
\email{m.villa-2 ``at'' sms.ed.ac.uk}
\curraddr{}
\email{}
\thanks{}

\date{}

\dedicatory{}
\begin{document}

\maketitle
\begin{center}

\begin{minipage}[c][][r]{300pt}
\begin{small}
\textsc{Abstract.} Let $\mu$ be a Radon measure on the $n$-th Heisenberg group $\bH^n$. In this note we prove that if the $(2n+1)$-dimensional (Heisenberg) Riesz transform on $\bH^n$ is $L^2(\mu)$-bounded, and if $\mu(F)=0$ for all Borel sets with $\dim_H(F)\leq 2$, then $\mu$ must have $(2n+1)$-polynomial growth. This is the Heisenberg counterpart of a result of Guy David from \cite{david-wavelets}.
\end{small}
\end{minipage}
\end{center}
\blfootnote{\textup{2010} \textit{Mathematics Subject Classification}: \textup{28A75}, \textup{28A12} \textup{28A78}.

\textit{Key words and phrases.} Rectifiability, Calder\'{o}n-Zygmund theory, singular integrals, Heisenberg group. 
}
\tableofcontents

\section{Introduction}
The motivation behind this note is the following question: what are the measures $\mu$ on the Heisenberg group $\bH^n$ which guarantee that the (correct notion of) Riesz transform is bounded from $L^2(\mu)$ to itself? This question (or some variant of it) with $\R^n$ instead of $\bH^n$, was one of the major starting point of the theory that came to be known as \textit{quantitative rectifiability}. This area of geometric measure theory has seen an impressive development in the past thirty years, starting with the landmark works of Peter Jones \cite{jones} and David and Semmes \cite{david-semmes91}, \cite{david-semmes93}, through the solution of fundamental questions in complex analysis, such as the Painlev\'{e} problem (see \cite{mmv}, \cite{david98}, \cite{tolsa03}), to more recent applications to harmonic analysis, see for example \cite{ntv} and \cite{ahmmt}.

In the last years, there has been an increasing interest in developing such a quantitative theory in different contexts than that of Euclidean spaces; examples of these are parabolic spaces and Heisenberg groups, or, more generally, Carnot groups. The former appear in the study of caloric measure. The latter arise naturally in the study of certain hypoelliptic operators, in the sense that the natural translations and dilations for these operators are those characterising the spaces; the Heisenberg group is the most important prototypical example, and the related operator is the so-called Kohn Laplacian; see \cite{blu} for a comprehensive study of stratified Lie groups and the corresponding operators. 

We should mention that the study of Heisenberg geometry can be approached from different perspectives and with different applications in mind; for example, see \cite{ny} for a connection with theoretical computer science. 

To be a little more specific: the starting motivation to develop a theory of quantitative rectifiability connected to our initial question, is to understand basic issues such as the removable sets for harmonic function (with respect to the relevant sub-Laplacian), or to give a characterisation of those domains where the Dirichlet problem (again, for the relevant sub-Laplacian) is well-posed. We want to underline, however, that a theory of quantitative rectifiability in the Heisenberg setting has its own, purely geometric, intrinsic appeal.

In the last couple of years, there has been some progress towards an answer to our initial question; see for example \cite{cfo}, \cite{fo18} and \cite{o18}. In this note we give a necessary condition to be imposed on a Radon measure $\mu$ on $\bH^n$ for the Riesz transform to be $L^2(\mu)$ bounded. Here $R_\mu$ is the singular integral operator whose kernel is the horizontal gradient of the fundamental solution of the Heisenberg sub-Laplacian, as defined in \cite{chousionis2012singular}. See Section \ref{preliminaries} for precise definitions.  

\begin{theorem} \label{t:main}
Let $\mu$ be a Radon measure on $\bH^n$ such that $R_\mu$ is bounded on $L^2(\mu)$ with norm $C_1$, and such that $\mu(F) = 0$ whenever $\dim_H(F) \leq 2$. Then there exists a constant $C_2$ such that for all balls $B(x,r) \subset \bH^n$, we have
\begin{align}\label{e:main}
    \mu(B(x,r)) \leq C_2 r^{2n +1}.
\end{align}
Here $C_2$ depends only on $n$ and $C_1$, and the ball $B(x,r)$ is defined with respect to the Kor\'{a}nyi metric, see Section \ref{preliminaries}.
\end{theorem}

A corresponding statement holds in the Euclidean setting, and is a result of David, \cite{david-wavelets}, Part III, Proposition 1.4. See  \cite{orponen-notes}, Proposition 6.9 for a more detailed proof. Let $\mathcal{R}^{d}_\mu$ denote the standard $d$-dimensional Riesz transform in $\R^n$.
\begin{theorem}
Assume that $\mu$ is a non-atomic Radon measure on $\R^n$ such that $\mathcal{R}^{d}_\mu$ is bounded on $L^2(\mu)$ with norm $C_1$. Then, for all Euclidean balls $B_{\R^n}(x,r)\subset\R^n$ we have
\begin{align}\label{e:growth}
\mu (B_{\R^n}(x,r)) \leq C_2 r^d    
\end{align}
Here $C_2$ depends only on $C_1$, $n$, and $d$. 
\end{theorem}
A measure satisfying \eqref{e:growth} (or \eqref{e:main}) is said to have \textit{polynomial growth}. 
Let us give a couple of remarks. 
\begin{remark}
Although the result itself (both in the Euclidean and Heisenberg case) is neither hard nor deep, it is nevertheless very useful. For example, most tools developed in the last two decades that take quantitative rectifiability beyond Ahlfors regular measures still need polynomial growth\footnote{With some exceptions, see for example \cite{azzam-schul18}, or \cite{badger-schul15}.} (see for example the book by Tolsa \cite{tolsa-book}). Thus, we expect that our result will be quite useful, too. 
\end{remark}
\begin{remark}
While the two results above look similar, there is actually a difference, in the sense that, in the Heisenberg case, there actually exist lower dimensional measures which give a bounded Riesz transform, but are not atomic.

This is \textit{not} a byproduct of the proof, but rather a fact of the Heisenberg geometry. Indeed, the $2$-dimensional $t$-axis (or any Heisenberg translate of it) gives a bounded $(2n+1)$-dimensional Riesz transform; this is simply because on these sets the kernel vanishes identically, see \eqref{e:K-norm}.

One can construct a more interesting example in the vertical plane of the one dimensional Heisenberg group $\bH$, say. Consider a tube of height $1$ and radius $\varepsilon_1^2$ around the $t$-axis, and take the intersection with the vertical plane. Call the resulting rectangle $R_{1,1}$. Cut out from $R_1$ two smaller rectangles $R_{2,1}$ and $R_{2,2}$, one in the top right corner and one in the bottom left corner, both of height $\varepsilon_2$ and width $\varepsilon_2^2$, for some $\varepsilon_2\le \varepsilon_1/4$. We proceed in this way, so that after $k$ steps we have $2^{k-1}$ disjoint rectangles $\{R_{k,i}\}_i$ of height $\varepsilon_k$ and width $\varepsilon_k^2$. Consider the natural probability measure $\mu$ on the Cantor-like set $C=\bigcap_k\bigcup_i R_{k,i}$. It is not difficult to show that, if $\varepsilon_k\to 0$ are small enough, the Heisenberg Riesz transform is bounded on $L^2(\mu)$; the idea is that the set is concentrated along the $t$-axis, and thus the kernel is very small (see \eqref{e:K-norm} below). Depending on the choice of $(\varepsilon_k)$ we have $\dim_H(C)\in [0,2].$
\end{remark}

\subsection*{Organisation of the article} In Section \ref{preliminaries} we briefly recall basic facts about Heisenberg groups and the Riesz transform. We also introduce a family of ``dyadic cubes'' suitable to our setting. 

Section \ref{s:main lemma} is dedicated to Lemma \ref{l:main lemma}, our main technical lemma. Roughly speaking, we show that if a measure $\mu$ is such that $R_{\mu}$ is bounded on $L^2(\mu)$, and there is some cube $Q_0$ with a very high concentration of $\mu$ (i.e. $\mu(Q_0)\gg \ell(Q_0)^{2n+1}$), then we can find a family $\HD(Q_0)$ of much smaller cubes, contained in $Q_0$, such that
\begin{enumerate}
    \item[a)] a very large portion of measure  $\mu$ on $Q_0$ is concentrated on the cubes from $\HD(Q_0)$,
    \item[b)] the family $\HD(Q_0)$ is relatively small, in the sense that it consists of few cubes.
\end{enumerate}

In Section \ref{s:iteration} we show that if the polynomial growth condition \eqref{e:main} is not satisfied, then we can find a cube satisfying the assumptions of our main lemma. This in turn allows us to start an iteration algorithm, consisting of using the main lemma countably many times, that results in constructing a set $Z$  with $\mu(Z)>0$ and $\dim_H(Z)\le 2.$ This finishes the proof of Theorem \ref{t:main}.

\subsection*{Acknowledgements}
We thank Katrin F\"{a}ssler and Tuomas Orponen for introducing us to the Heisenberg group and for suggesting the problem. 

The bulk of this work was done while the two authors were attending the Simons Semester in Geometric Analysis at IMPAN in the Autumn 2019. We thank Tomasz Adamowicz and the staff at the institute for their hospitality.

\section{Preliminaries}\label{preliminaries}
In our estimates we will often use the notation $f\lesssim g$ which means that there exists some absolute constant $C$ for which $f\le Cg.$ If the constant $C$ depends on some parameter $t$, we will write $f\lesssim_t g.$ Notation $f\approx g$ will stand for $f\lesssim g\lesssim f,$ and $f\approx_t g$ is defined analogously. For simplicity, in our estimates we will suppress the dependence on dimension $n$ and on absolute constants $\lambda,\ \Lambda$ (see \eqref{e:balls in cubes}).

\subsection{Heisenberg group}
In this paper we consider the $n$-th Heisenberg group with exponential coordinates (see \cite{cdpt} or \cite{faessler-notes} for a swift introduction to the Heisenberg group in a context close to ours). In practice, we will denote a point $p \in \bH^n$ as $(z,t) \in \R^{2n} \times \R$, and $z=(x_1,...,x_n,y_1,...,y_n)$. In these coordinates the group law in $\bH^n$ takes the form 
\begin{align*}
    p \cdot q = \ps{z + z', t+t' + \frac{1}{2} \sum_{i=1}^n (x_i y_i' + y_i x_i')},
\end{align*}
where $p=(z,t)$ and $q=(z',t')$. The identity element is the origin $(0,0)$ and the inverse is given by $p^{-1}=(-z,-t)$. We make $\bH^n$ into a metric space by setting $ d(p,q) := \|q^{-1}\cdot p\|_\bH$, where 
\begin{align}\label{e:korany}
    \|p\|_\bH^4 := |z|^4 + 16t^2,
\end{align}
and $|z|$ denotes the Euclidean norm of $z\in\R^{2n}$.

Note that $\|\cdot\|_\bH$ is $1$-homogeneous with respect to the anisotropic dilation $p \mapsto \lambda p = (\lambda z, \lambda^2 t)$, $\lambda >0$. The metric $d$ is sometimes called the Kor\'{a}nyi metric.

Given $p\in\bH^n$ and $r>0$ we set
\begin{equation*}
    B(p,r)=\ck{q\ |\ d(p,q)\le r},\quad U(p,r) = \ck{q\ |\ d(p,q)< r}.
\end{equation*}
For $\alpha>0$ we will write $\mathcal{H}^{\alpha}$ to denote the usual $\alpha$-dimensional Hausdorff measure with respect to metric $d$. For $A\subset\bH^n$ we set $\dim_H(A)$ to be the Hausdorff dimension of $A$. 

It follows easily from the definition of the Kor\'{a}nyi metric that for all $p\in\bH^n$ and $r>0$ we have
\begin{equation}\label{e:measure of balls}
    \mathcal{H}^{2n+2}(B(p,r)) = \mathcal{H}^{2n+2}(B(0,1))\, r^{2n+2}.
\end{equation}
Thus, even though the topological dimension of $\bH^n$ is $2n+1$, the Hausdorff dimension of $\bH^n$ is equal to $2n+2$. For the sake of brevity we set $D:=2n+2.$ Usually one denotes the Hausdorff dimension of $\bH^n$ by $Q$, but we have decided to save that letter for cubes; hence the non-standard notation.

It is also easy to check that if $\mathcal{L}^{2n+1}$ denotes the usual Lebesgue measure on $\R^{2n+1}\simeq \bH^n,$ then we have a constant $C>0$ such that
\begin{equation}\label{e:lebesgue and hausdorff}
    \mathcal{L}^{2n+1}=C\mathcal{H}^{D}.
\end{equation}

\subsection{Heisenberg Riesz transform}
Recall that, for a function $u : \bH^n \to \R$, the horizontal gradient of $u$ is given by
\begin{align*}
    \grad_\bH u := \ps{ X_1 u,...,X_n u, Y_1 u, ..., Y_n u},
\end{align*}
where the vector fields $X_1,\dots,X_n,Y_1,\dots, Y_n$ and $\frac{\partial}{\partial t}$ represent the left invariant translates of the canonical basis at the identity. In particular, $X_1,\dots,X_n,Y_1,\dots, Y_n$ span the horizontal distribution in $\bH^n$. 

The Heisenberg sublaplacian $\Delta_\bH$ is given by $\sum_{i=1}^n X_i^2 + Y_i^2$, and its fundamental solution is 
\begin{align*}
    G(p) := c_n \|p\|_\bH^{2-D}.
\end{align*}

The $(D-1)$-dimensional Riesz kernel in $\bH^n$, first considered in \cite{chousionis2012singular}, is given by
    $K(p)=\grad_{\bH} G(p).$
The Riesz transform is formally defined as
\begin{align*}
    R_\mu f (p) = \int_{\mathbb{H}^n} K(q^{-1}\cdot p)f(q) \, d\mu(q).
\end{align*}
Since it is not clear whether the integral above converges, one considers the truncated Riesz transform given by the formula
\begin{align*}
    R_{\mu, \delta}f(p)= \int_{\bH^n \setminus B(p,\delta)} K(q^{-1}\cdot p)f(q)\ d\mu(q),
\end{align*}
for $\delta>0$. We say that $R_{\mu}$ is bounded on $L^2(\mu)$ if the truncated operators $R_{\mu, \delta}$ are bounded on $L^2(\mu)$ uniformly in $\delta>0.$

One can easily check that the Riesz kernel is actually equal to
\begin{align*}
    & K(z, t) \\
    & = n\, \ps{ \frac{ -2x_1 |z|^2 + 8y_1 t }{\|(z,t)\|_\bH^{2n+4
    }}, \,\cdots ,    \frac{ -2x_n |z|^2 + 8y_n t }{\|(z,t)\|_\bH^{2n+4
    }}, \, \frac{ -2y_1 |z|^2 - 8x_1 t }{\|(z,t)\|_\bH^{2n+4}}, \cdots , \frac{ -2y_n |z|^2 - 8x_n t }{\|(z,t)\|_\bH^{2n+4}}}.
\end{align*}
Hence,
\begin{align} \label{e:K-norm}
    |K(z,t)|^2 = n^2 \, \frac{ 4 |z|^2 }{(|z|^4 + 16t^2)^{n+1}}.
\end{align}
This implies the curious fact that 
    $|K(z,t)| \leq C$
whenever
\begin{align} \label{e:paraboloid}
    |z| \leq  \, 16 |t|^{n+1},
\end{align}
which is a `paraboloidal' double cone around $t$-axis with vertex at the origin. This fact will play a key role in the subsequent analysis. 

Chousionis and Mattila showed in \cite[Proposition 3.11]{chousionis2012singular} that the Riesz kernel is a standard kernel. In particular, it satisfies the following continuity property: whenever  $q_1, q_2\neq p\in \bH^n$, we have
\begin{equation*}\label{e:kernel continuous}
    |K(p^{-1}\cdot q_1) - K(p^{-1}\cdot q_2)| \lesssim \max\bigg\{\frac{d(q_1,q_2)}{d(p,q_1)^{D}},\frac{d(q_1,q_2)}{d(p,q_2)^{D}} \bigg\}.
\end{equation*}
Taking $p=0$ and $q_1 = \tilde{q_1}^{-1}\cdot \tilde{p},\ q_2 = \tilde{q_2}^{-1}\cdot \tilde{p},$ one gets immediately that for all  $\tilde{q_1}, \tilde{q_2}\neq \tilde{p}\in \bH^n$
\begin{equation}\label{e:kernel continuous 2}
    |K(\tilde{q_1}^{-1}\cdot \tilde{p}) - K(\tilde{q_2}^{-1}\cdot \tilde{p})| \lesssim \max\bigg\{\frac{d(\tilde{q_1},\tilde{q_2})}{d(\tilde{p},\tilde{q_1})^{D}},\frac{d(\tilde{q_1},\tilde{q_2})}{d(\tilde{p},\tilde{q_2})^{D}} \bigg\}.
\end{equation}

\subsection{Dyadic cubes}
We are going to use a family of decompositions of $\bH^n$ into subsets that share many properties with the standard dyadic cubes from $\R^n$. The most classical constructions of this kind are due to Chirst \cite{christ1990b} and David \cite{David1988}, but for us it will be more convenient to use the ``cubes'' constructed in \cite{kaenmaki2012}.

First, note that given any ball $B(p,2r)$, one may use the $5r$-covering lemma and the property \eqref{e:measure of balls} to conclude that there exists some absolute constant $m$ such that $B(p,2r)$ may be covered by $m$ balls $B(p_i,r)$, where $\ck{p_i}_{i=1}^m$ are points in $B(p,2r)$. That is, $\bH^n$ is geometrically doubling. In particular, we can use {\cite[Theorem 2.1, Remark 2.2]{kaenmaki2012}}.

\begin{lemma}[\cite{{kaenmaki2012}}] \label{l:cubes}
	For all $k\in\Z$ there exists a family of subsets of $\bH^n$, denoted by $\dD_k$, such that
	\begin{enumerate}
		\item $\bH^n = \bigcup_{Q\in\dD_k}Q$,
		\item if $k\ge l$, and $Q\in \dD_k,\ P\in\dD_l$, then either $Q\cap P=\varnothing$ or $Q\subset P$,
		\item for every $Q\in \dD_k$ there exists $p_Q\in Q$ such that
		\begin{equation}\label{e:balls in cubes}
		U(p_Q,\lambda2^{-k})\subset Q\subset B(p_Q,\Lambda2^{-k})
		\end{equation}
		for some absolute constants $\lambda,\Lambda>0$.
	\end{enumerate}
\end{lemma}

Let us stress once more that we will not keep track of how various parameters appearing in the proof depend on $\lambda$ and $\Lambda$.

We set $\dD = \bigcup_k \dD_k$. For $Q\in\dD_k$ we define the sidelength of $Q$ as $\ell(Q)=2^{-k}$. Clearly, by \eqref{e:measure of balls} and \eqref{e:balls in cubes}, for $Q\in\dD$ we have
\begin{equation*}
    \mathcal{H}^{{D}}(Q) \approx \ell(Q)^{{D}}.
\end{equation*}
It follows that if $Q\in\dD$, then for $k\ge 0$
\begin{equation}\label{e:dyadic property}
\#\ck{P\in\dD\ |\ P\subset Q,\ \ell(P)=2^{-k}\ell(Q) }\approx 2^{kD}.
\end{equation}

Given a Radon measure $\mu$ and $Q\in\dD$ we will denote the $(D-1)$-dimensional density of $\mu$ in $Q$ by
\begin{equation*}
    \Theta_{\mu}(Q) = \frac{\mu(Q)}{\ell(Q)^{D-1}}.
\end{equation*}
For simplicity, we will suppress the dependence on $\mu$ and simply write $\Theta(Q)$.

\section{Main lemma}\label{s:main lemma}
   Our main tool in the proof of Theorem \ref{t:main} is the following lemma.
	\begin{lemma}\label{l:main lemma}
		Let $\mu$ be a Radon measure on $\bH^n$ such that $R_{\mu}$ is bounded on $L^2(\mu)$ with norm $C_1$. There exist constants $A=A(n)>1,\ s=s(A,n)\in (0,1/2)$ and $M=M(C_1,n)>100$ such that the following holds.
		
		Suppose that $Q_0\in\dD$ satisfies $\Theta(Q_0)\ge M$. Set $N = \floor{A^{-2} \log(\Theta(Q_0))}$. Then, the family of high density cubes
		\begin{equation*}
		\HD(Q_0) = \ck{Q\in \dD\ |\ Q\subset Q_0,\ \ell(Q)=2^{-N}\ell(Q_0),\ \Theta(Q)>2\,\Theta(Q_0)}
		\end{equation*}
		satisfies
		\begin{equation}\label{e:HD are thicc}
		\sum_{Q\in\HD(Q_0)}\mu(Q) \ge (1-\Theta(Q_0)^{-s})\mu(Q_0).
		\end{equation}
		Moreover, we have
		\begin{equation}\label{e:HD packing}
		\sum_{Q\in\HD(Q_0)}\ell(Q)^2\le C_p\,  \ell(Q_0)^2
		\end{equation}
		for some dimensional constant $C_p$ (``$p$'' stands for ``packing'').
	\end{lemma}
	
	The rest of this section is dedicated to proving the lemma above. For brevity of notation, we set $\Theta_0=\Theta(Q_0)$. Observe that the integer $N$ was chosen in such a way that
	\begin{equation}\label{e:choice of N}
	2^{A^2N}\approx \Theta_0\ge M.
	\end{equation}
	In particular, we have $N\ge N_0$ for some very big $N_0$ depending on $M$ and $A$.
	%
	
	We split the proof of Lemma \ref{l:main lemma} into several steps.
	
	First, note that by the pigeonhole principle and \eqref{e:dyadic property}, we can find a cube $Q_1 \in \dD$ with sidelength $\ell(Q_1) = 2^{-AN} \ell(Q_0)$ such that 
	\begin{align}\label{e:assumption Q1}
	\mu(Q_1) \gtrsim \frac{\mu(Q_0)}{2^{AND}}.    
	\end{align}
	Without loss of generality, by applying the appropriate translation, we can assume that $Q_1$ is centred at the origin, i.e. $p_{Q_1}=0$.
	Set
	\begin{align*}
	T:=\ck{ (z,t) \in Q_0 \, |\, |z| \leq 2^{-N} \ell(Q_0) }
	\end{align*}
	and for any $\kappa>0$ set 
	\begin{equation*}
	{ T_{\kappa}}:=\ck{ (z,t) \in Q_0 \, |\, |z| \leq \kappa\, 2^{-N} \ell(Q_0) }.
	\end{equation*}
	Observe that $Q_1\subset T.$ In a sense, $T$ can be seen as a tube with vertical axis passing through $p_{Q_1}=0$. Note also that for any cube $Q\subset Q_0 \setminus T$ we have $\dist(Q,Q_1)\gtrsim 2^{-N}\ell(Q_0).$
	
	We start by proving a few preliminary results.
	\begin{lemma}\label{c:number-cubes}
	There are at most $C(\kappa)\, 2^{2N}$ cubes of sidelength $2^{-N}\ell(Q_0)$ contained in $T_\kappa$.
	\end{lemma}
	\begin{proof}
	Observe that since $0\in Q_0$, and by \eqref{e:balls in cubes} $Q_0\subset B(p_{Q_0},\Lambda\ell(Q_0))$, we have $Q_0\subset B(0,2\Lambda\ell(Q_0)).$ Hence,
	\begin{align*}
	    T_{\kappa} &\subset \ck{ (z,t)\in B(0,2\Lambda\ell(Q_0))\ |\ |z|\le \kappa\, 2^{-N} \ell(Q_0) }\\
	    &\subset \ck{ (z,t)\in \bH^n\ |\ |z|\le \kappa\, 2^{-N} \ell(Q_0),\ 16|t|^2\le (2\Lambda \ell(Q_0))^4 }=: \widetilde{T}_{\kappa}.
	\end{align*}
	By \eqref{e:lebesgue and hausdorff},
	\begin{equation*}
	    \mathcal{H}^{D}(\widetilde{T}_{\kappa}) = C\mathcal{L}^{2n+1}(\widetilde{T}_{\kappa}) \approx (\kappa 2^{-N} \ell(Q_0))^{2n}(2\Lambda \ell(Q_0))^2 \approx_{\kappa} 2^{-2nN}\ell(Q_0)^{D}.
	\end{equation*}
	It follows that $\mathcal{H}^{D}(T_{\kappa})\lesssim_{\kappa}2^{-2nN}\ell(Q_0)^{D}.$ On the other hand, recall that for any cube $Q$ with sidelength $\ell(Q)=2^{-N}\ell(Q_0)$ we have $\mathcal{H}^{D}(Q)\approx 2^{-ND}\ell(Q_0)^{D}$. Since all such cubes are pairwise disjoint, we get
	\begin{equation*}
	    \#\ck{Q\in\mathcal{D}\ |\ \ell(Q)=2^{-N}\ell(Q_0),\ Q\subset T_{\kappa} } \lesssim	\fr{\mathcal{H}^{D}(T_{\kappa})}{2^{-ND}\ell(Q_0)^{D}}\lesssim_{\kappa} \fr{2^{-2nN}\ell(Q_0)^{D}}{2^{-N(2n+2)}\ell(Q_0)^{D}}=2^{2N}.
	  \end{equation*}
	\end{proof}
	\begin{lemma}\label{c:dense-tube}
		Let $Q\in\dD$ satisfy $Q \subset Q_0 \setminus T$ and $\ell(Q) = \ell(Q_1)=2^{-AN}\ell(Q_0)$. Then
		\begin{equation*}
		    \mu(Q) \le \frac{\mu(Q_0)}{\Theta_0 \, 2^{AND}}.
		\end{equation*}
	\end{lemma}
	\begin{proof}
		Suppose the claim above is false. Then we can find a cube $Q_2 \subset Q_0 \setminus T$ with $\ell(Q_2)=2^{-AN}\ell(Q_0)$ such that
		\begin{equation}\label{e:assumption Q2}
		\mu(Q_2) \geq \frac{\mu(Q_0)}{\Theta_0 \,2^{AND}}.
		\end{equation}
		Let $0< \delta < \dist(Q_1, Q_2)$, let $p\in Q_2$ be arbitrary, and consider 
		\begin{equation*}
		R_{\mu,\delta}(\chara_{Q_1})(p) = \int_{Q_1} K(q^{-1} \cdot p) \, d\mu(q).   
		\end{equation*}
		By triangle inequality,
		\begin{equation}\label{e:split}
		|R_{\mu,\delta}(\chara_{Q_1})(p)| \geq \av{ \int_{Q_1} K(p) \, d\mu(q) } - \av{ \int_{Q_1} K(q^{-1}\cdot p) - K(p) \, d\mu(q) }.
		\end{equation}
		We estimate the first term as follows. Note that, since $p \in Q_2$ and $Q_2$ lies outside $T$, then, writing $p=(z, t)$ and using \eqref{e:K-norm}, we have
		\begin{align*}
		|K(p)|^2 \approx \frac{|z|^2}{(|z|^4 + 16t^2)^{n+1}} \gtrsim \frac{|z|^2}{\ell(Q_0)^{4(n+1)}} \geq 2^{-2N}\ell(Q_0)^{-4n-2} = 2^{-2N}\ell(Q_0)^{-2D+2}.
		\end{align*}
		And thus we also have 
		\begin{equation}\label{e:first term estimate}
		\av{ \int_{Q_1} K(p)\, d\mu(q) } = \av{K(p)}\mu(Q_1) \gtrsim 2^{-N} \, \frac{\mu(Q_1)}{\ell(Q_0)^{D-1}}.
		\end{equation}
		
		For the second term in \eqref{e:split} we use the continuity of the kernel $K$ \eqref{e:kernel continuous 2} and the fact that $d(p,q) \approx \|p\|_{\bH}\ge 2^{-N}\ell(Q_0)$ (because $p\in Q_2\subset Q_0\setminus T$):
		\begin{align}
		|K(q^{-1} \cdot p) - K(p)| \lesssim \fr{\|q\|_{\bH}}{\min(\|p\|_{\bH}, d(p,q))^{D}}\lesssim \fr{2^{-AN}\ell(Q_0)}{(2^{-N}\ell(Q_0))^{D}} = \fr{2^{-AN+DN}}{\ell(Q_0)^{D-1}}.
		\end{align}
		Taking $A\ge 2D$ we get
		\begin{equation*}
		\av{ \int_{Q_1} K(q^{-1} \cdot p) - K(p) \, d\mu(q) }\lesssim 2^{-AN/2}\fr{\mu(Q_1)}{\ell(Q_0)^{D-1}}.
		\end{equation*}
		Together with \eqref{e:first term estimate} and \eqref{e:split}, assuming $N_0$ bigger than some absolute constant (recall that $N\ge N_0$), this gives
		\begin{equation*}
		|R_{\mu,\delta}(\chara_{Q_1})(p)|\gtrsim 2^{-N} \, \frac{\mu(Q_1)}{\ell(Q_0)^{D-1}}
		\end{equation*}
		for all $p\in Q_2$.
		
		Now, we use the estimate above and the $L^2(\mu)$ boundedness of $R_{\mu}$ to get
		\begin{equation*}
		2^{-N}\frac{\mu(Q_1)}{\ell(Q_0)^{D-1}}\mu(Q_2)^{\frac{1}{2}} \lesssim \ps{\int |R_{\mu,\delta}(\chara_{Q_1})(p)|^2 \, d\mu(p)}^{\frac{1}{2}} \leq C_1 \mu(Q_1)^{\frac{1}{2}}.
		\end{equation*}
		Our assumptions on $Q_1$ \eqref{e:assumption Q1} and $Q_2$ \eqref{e:assumption Q2} yield
		\begin{align*}\label{e:est-cauchy}
		C_1 & \gtrsim 2^{-N}\frac{\mu(Q_1)^{\frac{1}{2}}\mu(Q_2)^{\frac{1}{2}}}{\ell(Q_0)^{D-1}}\gtrsim 2^{-N}\fr{\mu(Q_0)}{2^{AND}\ell(Q_0)^{D-1}}\Theta_0^{-1/2} = 2^{-AND-N}\Theta_0^{1/2} \notag\\
	& \overset{\eqref{e:choice of N}}{\approx} 2^{-AND-N}\, 2^{A^2 N/2}.
		\end{align*}
		Taking $A\ge 5D$ we can bound the last term from below in the following way:
		\begin{equation*}
		    2^{-AND-N+A^2 N/2}\ge  2^{A^2N/4}\overset{\eqref{e:choice of N}}{\gtrsim} M^{1/4}.
		\end{equation*}
		Putting together the estimates above gives $C_1\gtrsim M^{1/4}$, which is a contradiction for $M=M(C_1,n)$ big enough.
	\end{proof}
	
	We immediately get the following corollary.
	\begin{corollary}
	We have
	\begin{equation}\label{e:T2 thicc}
	\mu({T_2})\ge (1-\Theta_0^{-1})\mu(Q_0).
	\end{equation}
	\end{corollary}
	\begin{proof}
	Observe that if $Q\in\mathcal{D}$ satisfies $\ell(Q) = \ell(Q_1) = 2^{-AN}\ell(Q_0)$ and $Q\not\subset T_2$, then we have $Q\cap T=\varnothing$ (assuming $A$ large enough with respect to $\Lambda$). It follows that $Q$ satisfies the assumptions of Lemma \ref{c:dense-tube}, and so
	\begin{equation*}
	    \mu(Q)\le2^{-AND}\Theta_0^{-1}\mu(Q_0).
	\end{equation*}
	Summing over all such $Q$ and using \eqref{e:dyadic property} yields
	\begin{equation*}
	    \mu(Q_0\setminus T_2) \le \Theta_0^{-1}\mu(Q_0).
	\end{equation*}
	\end{proof}
	Recall that
	\begin{equation*}
	\HD(Q_0) = \ck{Q\in \dD\ |\ Q\subset Q_0,\ \ell(Q)=2^{-N}\ell(Q_0),\ \Theta(Q)>2\Theta_0  },
	\end{equation*}
	and that $\Lambda$ is the absolute constant such that $Q\subset B(p_Q,\Lambda\ell(Q)).$ Without loss of generality, we may assume $\Lambda>2$.
	
	We are ready to prove the first part of Lemma \ref{l:main lemma}, the estimate \eqref{e:HD are thicc}.
	\begin{lemma}\label{l:measureHD}
		There exists $s=s(A,n)\in (0,1/2)$ such that
		\begin{equation}\label{e:HD are thicc 2}
		\sum_{Q\in\HD(Q_0)}\mu(Q) \ge (1-\Theta_0^{-s})\mu(Q_0).
		\end{equation}
	\end{lemma}
	\begin{proof}
		We will prove \eqref{e:HD are thicc 2} by contradiction. Suppose that
		\begin{equation}\label{e:HD not thicc :(}
		\sum_{Q\in\HD(Q_0)}\mu(Q) < (1-\Theta_0^{-s})\mu(Q_0).
		\end{equation}
		Set
		\begin{equation*}
		\LD(Q_0) = \ck{Q\in\dD\ |\ Q\subset T_{2\Lambda},\ \ell(Q)=2^{-N}\ell(Q_0),\ \Theta(Q)\le 2\Theta_0  }.
		\end{equation*}
		It is easy to see that the cubes from $\HD(Q_0)\cup \LD(Q_0)$ cover $T_{2}$. If we assume $\Theta_0\ge M>100$, and $s<1/2$, then $\Theta_0^{-s}/2\ge \Theta_0^{-1}$, and so by \eqref{e:T2 thicc} and \eqref{e:HD not thicc :(} we get
		\begin{equation}\label{e:LD has plenty mass}
		\sum_{Q\in\LD(Q_0)}\mu(Q) \ge \fr{\Theta_0^{-s}}{2}\mu(Q_0).
		\end{equation}
		
		On the other hand, recall from Lemma \ref{c:number-cubes} that there are at most $C2^{2N}$ cubes of sidelength $2^{-N}\ell(Q_0)$ contained in $T_{2\Lambda}$, where $C=C(\Lambda,n)$. Moreover, for any $Q\in \LD(Q_0)$ we have
		\begin{equation*}
		\mu(Q)\le 2 \Theta_0 \ell(Q)^{D-1}= 2\, \mu(Q_0)\fr{\ell(Q)^{D-1}}{\ell(Q_0)^{D-1}} =2^{-N(D-1)+1}\mu(Q_0).
		\end{equation*}
		In consequence,
		\begin{equation*}
		\sum_{Q\in\LD(Q_0)}\mu(Q) \le C 2^{2N} 2^{-N(D-1)+1}\mu(Q_0).
		\end{equation*}
		This contradicts \eqref{e:LD has plenty mass} because
		\begin{equation*}
		C\, 2^{-ND+3N+1} = 2\, C\,  (2^{-A^2N})^{(-D+3)A^{-2}} \overset{\eqref{e:choice of N}}{\le}\widetilde{C}(n)\Theta_0^{(-D+3)A^{-2}}\le \fr{\Theta_0^{-s}}{2},
		\end{equation*}
		choosing $s=s(A,n)$ small enough.
	\end{proof}
	We move on to the second part of Lemma \ref{l:main lemma}, i.e. the packing estimate \eqref{e:HD packing}.
	\begin{lemma}\label{l:HD-in-tube}
	We have
	\begin{equation}\label{e:HD-in-tube}
	    \bigcup_{Q\in\HD(Q_0)}Q\subset T_{2\Lambda}.
	\end{equation}
	In consequence, 
	\begin{equation}\label{e:HD packing 2}
		\sum_{Q\in\HD(Q_0)}\ell(Q)^2\lesssim \ell(Q_0)^2.
		\end{equation}
	\end{lemma}
	\begin{proof}
	We will prove that for $Q\in\HD(Q_0)$ we have $Q\cap T_2\neq\varnothing.$ Then, since $\ell(Q)=2^{-N}\ell(Q_0)$, it follows easily from \eqref{e:balls in cubes} that indeed $Q\subset T_{\Lambda+2}(Q_0)\subset T_{2\Lambda}(Q_0)$.
	
	We argue by contradiction. Suppose that $Q\in \HD(Q_0)$ and $Q\cap T_2=\varnothing$. Consider the cubes $\{P_i\}_{i\in I}$ with $\ell(P_i)=2^{-AN}\ell(Q_0) = 2^{-(A-1)N}\ell(Q)$ and $P_i\subset Q$. Then, $Q=\bigcup_i P_i,$ for all $i\in I$ we have $P_i\cap T_2=\varnothing,$ and $\# I\approx 2^{(A-1)ND}$ by \eqref{e:dyadic property}. 
	
	We use Lemma \ref{c:dense-tube} to conclude that for all $i\in I$
	\begin{equation*}
	    \mu(P_i) \le \frac{\mu(Q_0)}{\Theta_0 \, 2^{AND}}.
	\end{equation*}
	Summing over $i\in I$ yields
	\begin{equation*}
	    \mu(Q) = \sum_{i\in I}\mu(P_i)\le \# I\cdot \frac{\mu(Q_0)}{\Theta_0 \, 2^{AND}} \approx 2^{(A-1)ND} \frac{\mu(Q_0)}{\Theta_0 \, 2^{AND}} = \frac{\mu(Q_0)}{\Theta_0 \, 2^{ND}},
	\end{equation*}
	so that
	\begin{equation*}
	    \Theta(Q) = \fr{\mu(Q)}{(2^{-N}\ell(Q_0))^{D-1}} \lesssim \frac{\mu(Q_0)}{\Theta_0 \, 2^{ND}}\cdot \frac{1}{2^{-N(D-1)}\ell(Q_0)^{D-1}} = \frac{\Theta_0}{\Theta_0 \, 2^{N}}=2^{-N}\le 1.
	\end{equation*}
	But this contradicts the assumption $Q\in \HD(Q_0)$:
	\begin{equation*}
	    \Theta(Q)\ge 2\Theta_0\ge 2M>1,
	\end{equation*}
	and so the proof of \eqref{e:HD-in-tube} is finished.
	
	Concerning \eqref{e:HD packing 2}, note that by \eqref{e:HD-in-tube} and Lemma \ref{c:number-cubes} we have
	\begin{equation} \label{e:HD-card}
	    \# \HD(Q_0) \lesssim 2^{2N}.
	\end{equation}
	Hence,
	\begin{align*}
	    \sum_{Q \in \HD(Q_0)} \ell(Q)^2 = \ell(Q_0)^2\,  2^{-2N}  \sum_{Q \in \HD(Q_0)} 1  \lesssim \ell(Q_0)^2. 
	\end{align*}
	\end{proof}
\section{Iteration argument}\label{s:iteration}
To complete the proof of Theorem \ref{t:main}, we assume that the measure $\mu$ does not satisfy the polynomial growth condition \eqref{e:main}. Then we will use Lemma \ref{l:main lemma} countably many times to construct a set $Z$ with positive $\mu$-measure and with Hausdorff dimension at most $2$. 

Suppose that there exists a ball $B(x,r)$ with $\mu(B(x,r)) \geq C_2 r^{2n+1}$; if $C_2$ is big enough, we can find a cube $Q_0 \in \dD,\ Q\subset B(x,r)$ such that 
	\begin{align*}
	\Theta(Q_0) \ge M, 
	\end{align*}
	where $M$ is the constant from Lemma \ref{l:main lemma}.

Let $A>1$ be as in Lemma \ref{l:main lemma}. Following the notation of Lemma \ref{l:main lemma}, for an arbitrary cube $Q\in\dD$ with $\Theta(Q)\ge M$, set
\begin{equation*}
    N(Q) := \floor{A^{-2} \log(\Theta(Q))}
\end{equation*}
and
\begin{equation*}
    \HD(Q):=\left\{ P \in \dD \, |\, P\subset Q,\, \ell(P) =2^{-N(Q)}\ell(Q), \, \Theta(P) > 2 \Theta(Q) \right\}.
\end{equation*}
Put $Z_{0}:=Q_0,\ \HD_0:=\{Q_0\},\ \HD_1:=\HD(Q_0),$ and $Z_1:=\bigcup_{Q\in\HD_1} Q$. Proceeding inductively, for all $j\ge 2$ we define
\begin{gather*}
    \HD_{j} := \bigcup_{Q \in \HD_{j-1}} \HD(Q),\\
    Z_{j} := \bigcup_{Q\in\HD_j} Q.
\end{gather*}
Note that for each $j$ the cubes in $\HD_j$ form a disjoint family. Moreover, $\{Z_j\}_{j\ge 0}$ form a decreasing sequence of sets, that is $Z_{j+1}\subset Z_j.$ Define
\begin{equation*}
   Z := \bigcap_{j \geq 0} Z_j.
\end{equation*}
\begin{claim}\label{c:positive}
We have
\begin{equation*}
    \mu(Z)\gtrsim_{M,s}\mu(Q_0).
\end{equation*}
\end{claim}
\begin{proof}
Observe that for $Q\in\HD_j$ we have
\begin{equation}\label{e:estimate densities}
    \Theta(Q)\ge 2^{j}\Theta(Q_0)\ge 2^{j}M.
\end{equation}
In particular, $\Theta(Q)\ge M$ and so we may apply Lemma \ref{l:main lemma} to $Q$. It follows that for any $j \geq 0$ we have
\begin{multline*}
    \mu(Z_{j+1})=\sum_{Q \in \HD_{j+1}} \mu(Q) = \sum_{Q \in \HD_{j}} \sum_{P \in \HD(Q)} \mu(P) \stackrel{\eqref{e:HD are thicc}}{\geq} \sum_{Q \in \HD_{j}} (1-\Theta(Q)^{-s})\mu(Q)\\
    \overset{\eqref{e:estimate densities}}{\ge} \sum_{Q \in \HD_{j}} (1-2^{-js}M^{-s})\mu(Q) = (1-2^{-js}M^{-s}) \mu(Z_{j}).
\end{multline*}
Using this estimate $(j+1)$ times we arrive at
\begin{equation}\label{e:HD_jbig}
    \mu(Z_{j+1}) \geq \prod_{i=0}^{j}(1-2^{-is}M^{-s}) \mu(Q_0).
\end{equation}
Since $Z_j$ form a sequence of decreasing sets, we get by the continuity of measure 
\begin{equation*}
    \mu(Z) = \lim_{j\to\infty} \mu(Z_j)\ge \prod_{i=0}^{\infty}(1-2^{-is}M^{-s}) \mu(Q_0) = C(s,M)\mu(Q_0),
\end{equation*}
where $C(s,M)$ is positive and finite because $\sum_{i=0}^{\infty} 2^{-is}<\infty$.
\end{proof}
\begin{claim}\label{c:dim} We have
\begin{equation*}
    \textstyle{\dim_H(Z)} \leq 2.
\end{equation*}
\end{claim}
\begin{proof}
Recall that $N(Q) = \floor{A^{-2} \log(\Theta(Q))}$. It follows from \eqref{e:estimate densities} that for $Q\in\HD_j$ we have $N(Q)\ge C_3jA^{-2}$ for some absolute constant $C_3>0$. Thus, for $Q\in\HD_j$ and $P\in\HD(Q)$
\begin{equation*}
    \ell(P)=2^{-N(Q)}\ell(Q)\le 2^{-C_3jA^{-2}}\ell(Q).
\end{equation*}
Using this observation $j$ times we get that for $P\in\HD_{j+1}$
\begin{equation*}
    \ell(P)\le 2^{-C_4 j(j+1)A^{-2}}\ell(Q_0),
\end{equation*}
where $C_4=C_3/2.$
Hence, the cubes from $\HD_j$ form coverings of $Z$ with decreasing diameters, well suited for estimating the Hausdorff measure of $Z$. 

Let $0<\varepsilon<1,\ 0<\delta<1$ be small. Let $j\ge 0$ be so big that for $Q\in\HD_{j}$ we have $\diam(Q)\le\Lambda\ell(Q)\le\delta$. Then,
\begin{equation}\label{e:Hausdorff estimate}
   \mathcal{H}_{\delta}^{2+\ve}(Z) \le \Lambda^{2+\varepsilon}\sum_{Q\in\HD_j}\ell(Q)^{2+\varepsilon}\le \Lambda^{2+\varepsilon}(2^{-C_4j(j-1)A^{-2}}\ell(Q_0))^{\varepsilon}\sum_{Q\in\HD_j}\ell(Q)^{2}.
\end{equation}
It follows by \eqref{e:HD packing} that
\begin{equation*}
    \sum_{Q\in\HD_j}\ell(Q)^{2} = \sum_{P\in\HD_{j-1}}\sum_{Q\in\HD(P)}\ell(Q)^{2}\le C_p\sum_{P\in\HD_{j-1}}\ell(P)^{2}.
\end{equation*}
Using the estimate above $j$ times, and putting it together with \eqref{e:Hausdorff estimate} we arrive at
\begin{equation*}
    \mathcal{H}_{\delta}^{2+\ve}(Z) \le \Lambda^{2+\varepsilon} (C_p)^j\, \,2^{-\varepsilon C_4 j(j-1)A^{-2}}\,  \ell(Q_0)^{2+\varepsilon}.
\end{equation*}
The right hand side above converges to 0 as $j\to\infty$ (just note that the exponent at $C_p$ is linear in $j$ while the exponent at $2$ is quadratic in $j$). Hence, $\mathcal{H}_{\delta}^{2+\ve}(Z)=0.$ Letting $\delta\to 0$ we get $\mathcal{H}^{2+\ve}(Z)=0.$ Since this is true for arbitrarily small $\ve>0$, it follows that 
\begin{equation*}
    \textstyle{\dim_H(Z)}  = \inf\{t\ge 0\ :\ \mathcal{H}^t(Z)=0\} \le 2.
\end{equation*}
\end{proof}
\begin{proof}[Proof of Theorem \ref{t:main}]
We have found a set $Z \subset \mathbb{H}^{n}$ of dimension smaller than or equal to $2$ (Claim \ref{c:dim}) but which nevertheless has positive $\mu$-measure (Claim \ref{c:positive}). This contradicts the assumptions of Theorem \ref{t:main}. Thus, there exists $C_2=C_2(n,C_1)$ such that $\mu(B(x,r)) \leq C_2 r^{2n+1}$ for all $x\in\bH^n$ and $r>0$.
\end{proof}

\subsection*{Funding}
D. D\k{a}browski was supported by Spanish Ministry of Economy and Competitiveness, through the Mar\'ia de Maeztu Programme for Units of Excellence in R\&D (grant MDM-2014-0445), and also partially supported by the Catalan Agency for Management of University and Research Grants (grant 2017-SGR-0395), and by the Spanish Ministry of Science, Innovation and Universities (grant MTM-2016-77635-P).

M. Villa was supported by The Maxwell Institute Graduate School in Analysis and its
Applications, a Centre for Doctoral Training funded by the UK Engineering and Physical
Sciences Research Council (grant EP/L016508/01), the Scottish Funding Council, Heriot-Watt
University and the University of Edinburgh.

Both authors were partially supported by the grant 346300 for IMPAN from the Simons Foundation and the matching 2015-2019 Polish MNiSW fund.

\bibliography{bibliography}
\bibliographystyle{halpha-abbrv}
\end{document}